\documentclass[letterpaper, 10 pt, conference]{ieeeconf}  

\IEEEoverridecommandlockouts                              
\usepackage{graphics} 
\usepackage{epsfig} 

\usepackage{mathptmx} 
\DeclareMathAlphabet{\mathcal}{OMS}{cmsy}{m}{n}
\SetMathAlphabet{\mathcal}{bold}{OMS}{cmsy}{b}{n}

\usepackage{times} 

\usepackage{amsmath} 
\usepackage{amssymb}  
\usepackage{algorithm}
\usepackage{algpseudocode}
\usepackage{mathtools}
\usepackage{xcolor}

\usepackage{ntheorem}
\newtheorem{assumption}{Assumption}
\newtheorem{remark}{Remark}
\newtheorem{problem}{Problem}

\newtheorem{lemma}{Lemma}
\newtheorem{corollary}{Corollary}[lemma]
\newtheorem{proposition}{Proposition}

\usepackage{url}

\usepackage{cite}
\usepackage{stfloats}

\allowdisplaybreaks[2]

\newcommand{\comment}[1]{}
\newcommand{\st}{\mathrm{s.t.}}

\title{\LARGE \bf 
Cubature-Based Statistical Moment Steering for Nonlinear, Non-Gaussian Trajectory Optimization
}

\author{Daniel C. Qi and Kenshiro Oguri
\thanks{This work was supported by the U.S. Air Force Office of Scientific Research through grant FA9550-23-1-0512.}
\thanks{D. C. Qi is a PhD student with the School of Aeronautics and Astronautics, Purdue University, West Lafayette, Indiana, 47907, USA. {\tt\small qi85@purdue.edu}}
\thanks{K. Oguri is an Assistant Professor with the School of Aeronautics and Astronautics, Purdue University, West Lafayette, Indiana, 47907, USA.}
}

\begin{document}

\maketitle
\thispagestyle{empty}
\pagestyle{empty}

\begin{abstract}

This paper addresses discrete-time non-Gaussian distribution steering in deterministic nonlinear systems. Distribution steering is a problem in which a control policy is designed over a distribution of trajectories rather than optimizing a single trajectory. The challenge with nonlinear systems is that even initially Gaussian distributions evolve into non-Gaussian distributions, often precluding a closed-form representation of the probability density. A method called statistical moment steering applies a cubature-based method to approximate non-Gaussian distributions for an approximate solution to this problem. This paper provides theoretical support and details how it can be solved with sequential convex optimization. A coupled, nonlinear, second-order ordinary differential equation describing a two-dimensional oscillator is provided as a numerical example.

\end{abstract}


\section{Introduction}
This paper addresses the problem of non-Gaussian distribution steering in deterministic nonlinear systems in discrete time over a finite horizon. Distribution steering is a problem in which a control policy is designed for an entire distribution of trajectories, rather than a single discrete trajectory. One of the simplest cases is the consideration of Gaussian distributions in linear systems. This type of problem is known as covariance control \cite{Hotz-CS, Collins-CS} or covariance steering (CS) more recently \cite{Okamoto-CS}. Since a multivariate Gaussian distribution is fully characterized by its mean vector and covariance matrix and remains Gaussian under affine transformations, the CS problem reduces to designing control policies that steer the evolution of these moments. Both continuous \cite{Chen-CS} and discrete versions \cite{Collins-CS, Bakolas-CS} of this problem exist. For the discrete version, the problem can be solved with convex programming \cite{Liu-CS}, or sequential convex programming for a square root covariance approach \cite{Naoya-Sqrt-CS}. 

CS provides theoretical guarantees only for linear Gaussian systems, but this technique has still been effectively demonstrated in nonlinear problems \cite{Ridderhof-NL-CS, Oguri-Safe-Autonomy, Naoya-CS-JGCD, Fife-CS, Benedikter-JGCD}. These approaches assume that the nonlinearly transformed distribution remains well approximated by a Gaussian. But ultimately, nonlinear systems inevitably result in non-Gaussian distributions, and the CS framework has been observed to break down in highly nonlinear engineering applications \cite{Qi-min-NL-Journal}. This motivates the problem of non-Gaussian distribution steering. 

The Liouville equation \cite{Zwanzig-StatMech, Boussouira-PDE-control} is a partial differential equation (PDE) that governs the evolution of the probability density function in a deterministic dynamical system. Although numerical methods can propagate a non-Gaussian probability density by solving this PDE, embedding it within a control optimization is computationally demanding because each control iterate requires a new solution to the PDE \cite{Boussouira-PDE-control}. To obtain tractable formulations, existing approaches focus on optimizing lower-dimensional statistical summaries of the distribution. For example, in linear systems, linear Gaussian Mixtures \cite{Naoya-NG-CS} and characteristic functions \cite{Sivaramakrishnan-CF-CS} have been used to design corresponding steering policies of non-Gaussians. But extending these methods to nonlinear systems may still be difficult, as they also require the nonlinear evolution of the statistical representation.

The uncertainty-quantification literature provides computationally tractable methods for approximating the evolution of non-Gaussian distributions through nonlinear systems, including nonlinear Gaussian-mixture models \cite{DeMars-GMM} and polynomial chaos expansions \cite{Jones-PCE}. This paper adopts cubature methods from uncertainty quantification to develop a corresponding control policy. Cubature points are discrete points used to approximate a distribution's statistical moments after a nonlinear transformation. A common cubature rule is the unscented transform, which has been applied to the nonlinear distribution steering problem \cite{Ozaki-UT-TrajOp, Ross-CUT-Traj-Op}, along with higher-order cubature rules \cite{Nandi-CUT}. Existing applications of these methods have been limited to mean and covariance control, whereas prior work by the authors demonstrated control of higher-order moments \cite{Qi-Stat-Moment-Steering}. This approach is referred to as statistical moment steering, and the present work establishes its theoretical foundation.

The contributions of this paper are as follows: (i) establish a theoretical connection between the Liouville equation and cubature-based statistical moment steering, (ii) describe statistical moment steering in a general framework using cubature points, and (iii) present a computationally efficient method of solving the steering problem using sequential convex programming.

\section{Preliminaries}\label{sec: prelim}
\subsection{Notation}
$\mathbb{E}[\cdot]$ denotes the expectation operator and $\mathcal{N}(\boldsymbol{\mu},P)$ denotes a normal distribution with mean $\mu$ and covariance $P$. $\mathbb{R}, \mathbb{R}^n, \mathbb{R}^{n\times m}$ denotes real numbers, $n$-dimensional real-valued vector, and $n\times m$ real-valued matrix respectively. $I_n$ denotes an identity matrix of size $n$. $|\cdot|$ denotes the absolute value. The notation $\mathbb{Z}_{a:b}$ denotes the set of integers between and including $a$ and $b$. $\mathrm{tr}(\cdot)$ and $\lambda_{\max}(\cdot)$ denote the trace and maximum eigenvalue of a matrix, respectively. ($\nabla \cdot$) denotes the divergence of a vector field. The square-root of matrix $A$ is $A=(A^{1/2})(A^{1/2})^\top$.

\subsection{Dynamical Systems and Flow Maps}
Let $\boldsymbol{x} \in \mathbb{R}^{n_x}$ denote the state of an ordinary differential equation with vector field $f$ and control $\boldsymbol{u} \in \mathbb{R}^{n_u}$:
\begin{equation}\label{eq: ODE}
\dot{\boldsymbol{x}}=f(\boldsymbol{x},\boldsymbol{u})
\end{equation}
Let the flow map be denoted by $\varphi$, and it is a differentiable mapping such that  
\begin{equation}
    \boldsymbol{x}(t) = \varphi_t(\boldsymbol{x}_0; \Pi_{[t_0,t]})
\end{equation}
where $\boldsymbol{x}_0$ is the state at $t_0$ with control policy from $t_0$ to $t$ denoted by $\Pi_{[t_0,t]}$. The fundamental matrix $\Phi$ for general nonlinear systems can be expressed by the Jacobian of the flow map:
\begin{equation}
\begin{aligned}
    \Phi(t,t_0; \Pi_{[t_0,t]}) &= \frac{\partial \varphi_t(\boldsymbol{x}_0; \Pi_{[t_0,t]})}{\partial \boldsymbol{x}_0} 
\end{aligned}
\end{equation}
For brevity, write $\Pi$ for $\Pi_{[t_0,t]}$ whenever the relevant time interval is clear from context.

\subsection{Statistical Moments}
Given a univariate random variable ${X}\in\mathbb{R}$, the expected value of ${X}$ can be represented as
\begin{equation}
    \mathbb{E}[\mathcal{G}({X})] = \int_{\mathbb{R}} \mathcal{G}(x) p(x) \mathrm{d}x
\end{equation}
where $p(x)$ is the probability density function (pdf) of ${X}$ and $\mathcal{G}(\cdot)$ a measurable function of ${X}$. The expectation integral can be used to calculate specific moments of the univariate random variable as detailed in Table~\ref{tab: Statistical Moments}.  

\begin{table}[H]
\caption{Statistical Moments of Univariate Random Variables}
\label{tab: Statistical Moments}
\renewcommand{\arraystretch}{1.5} 
\setlength{\tabcolsep}{6pt}       
\centering
\begin{tabular}{|c|c|}
\hline
Statistical Moment & Expectation Calculation\\
\hline
\hline
$m$-th Raw Moment & $\mathbb{E}[{X}^m]$  \\
\hline
$m$-th Central Moment & $\mathbb{E}[({X} - \mu)^m]$   \\
\hline
$m$-th Standardized Moment & $\mathbb{E}[({X} - \mu)^m/\sigma^m]$ \\
\hline
\end{tabular}
\end{table}
Certain moments carry greater significance and have specific names, as outlined in Table~\ref{tab: Moments Names}.  
\begin{table}[H]
\caption{Terminology for Moments of Univariate Random Variables}
\label{tab: Moments Names}
\renewcommand{\arraystretch}{1.5} 
\setlength{\tabcolsep}{6pt}       
\centering
\begin{tabular}{|c||c|c|c|}
\hline
$m$& $m$-th Raw & $m$-th Central & $m$-th Standardized \\
\hline
1 & Mean ($\mu$) & - & -  \\
\hline
2 & - & Variance ($\sigma^2$) & - \\
\hline
3 & - & -  & Skewness ($\gamma$) \\
\hline
4 & - & -  & Kurtosis  ($\kappa$)\\
\hline
\end{tabular}
\end{table}
\begin{remark}
For multivariate random variables, the integrand of the expectation integral becomes a multidimensional function. Given a random vector $\boldsymbol{X}\in\mathbb{R}^{n}$, the mean is defined as a vector $\in\mathbb{R}^{n}$ and its covariance as a matrix $\in\mathbb{R}^{n\times n}$. However, standardized univariate moments such as skewness have no unique multivariate analogue; several definitions exist depending on the application \cite{Averous-skewness, Jammalamadaka-skewness}. The paper's interpretation of these parameters is given in Section~\ref{sec: High-ordered Moments}. 
\end{remark}
For brevity, let $f^{(m)}(\boldsymbol{X})$ be an arbitrary function that relates the multi-dimensional $\boldsymbol{X}$ to any statistical moment:
\begin{equation}
\mathbb{E}[f^{(m)}(\boldsymbol{X})]=
\int_{\mathbb{R}^n} f^{(m)}(\boldsymbol{x})p(\boldsymbol{x})\,\mathrm{d}\boldsymbol{x}
\end{equation}

\subsection{Cubature Approximation to Expectation Integral}
Cubature points, or sigma points, are discrete points with associated weights used to approximate the multidimensional expectation integral as a summation. Let $\boldsymbol{x}^{(i)}$ be the cubature points associated with the pdf $p(\boldsymbol{x})$:
\begin{equation}\label{eq: cubature rule}
\mathcal{Q}\left[p(\boldsymbol{x})\right]
\triangleq 
\{(\boldsymbol{x}^{(i)},w_i)\}_{i\in \mathbb{Z}_{1:n_s}}
\end{equation}
where $\mathcal{Q}$ denotes the chosen cubature rule. Then,
\begin{equation}\label{eq: sigma point sum}
    \mathbb{E}[\mathcal{G}(\boldsymbol{X})] \approx \sum^{n_s}_{i=1} w_i \mathcal{G}(\boldsymbol{x}^{(i)})
\end{equation}
where $n_s$ denotes the number of cubature points. For brevity, notation for summations denotes only the index (e.g., $\sum_i$). There are many approaches for computing cubature points, including the unscented transform \cite{Julier-UT}, the conjugate unscented transform \cite{Adurthi-CUT}, and higher-order unscented transforms \cite{Ponomareva-HO-UT}. This paper makes the following assumption regarding the selection of the points and weights:
\begin{assumption}\label{assum: cubature points assumption}
    The cubature weights are strictly positive ($w_i > 0$) and for any prescribed finite moment order $m$, a cubature rule of sufficiently high degree can be selected to reproduce all moments up to order $m$.
\end{assumption}
The assumption ensures that any degree-$m$ moment selected for control can be accurately represented by the chosen cubature rule. The positivity assumption simplifies the subsequent factorization by avoiding additional sign-dependent cases.

\section{Problem Statement}\label{sec: problem statement}
This paper considers a discrete-time, finite-horizon, non-Gaussian steering problem. Let the time horizon be discretized into $N$ nodes, forming $N-1$ segments. Let $k = 0,1,\ldots, N-1$ denote the index for time $t_k$. The state $\boldsymbol{X}_k \in \mathbb{R}^{n_x}$ is a multi-dimensional random variable, and let $\boldsymbol{x}_k$ be the realization of $\boldsymbol{X}_k$. Likewise $\boldsymbol{U}_k \in \mathbb{R}^{n_u}$ and $\boldsymbol{u}_k$ represent the stochastic and realized control actions respectively. The notation for mean $\boldsymbol{\mu}_k = \mathbb{E}[\boldsymbol{X}_k]$ and covariance $P_k = \mathbb{E}[(\boldsymbol{X}_k-\boldsymbol{\mu}_k)(\boldsymbol{X_k}-\boldsymbol{\mu}_k)^\top]$ are defined for the state. Consider an objective $J$ to be optimized. This paper makes the following assumptions regarding the control of the distribution:
\begin{assumption}\label{assump: initial distribution}
The initial distribution $\mathcal{D}_0$ is known. 
\begin{equation}\label{eq: DS optimization (Initial Distribution)} 
\boldsymbol{X}_{0}\sim\mathcal{D}_0
\end{equation}
\end{assumption}
\begin{assumption}\label{assump: dynamics determinsitic}
The dynamics are continuous and deterministic. On each interval $[t_k,t_{k+1})$, the control action is fully specified by $u_k$: either throughout the interval (zero-order hold) or applied at $t_k$, followed by uncontrolled evolution (impulsive control).
\end{assumption}
Under Assumption~\ref{assump: dynamics determinsitic}, let the dynamics induce the flow map
\begin{equation}\label{eq: DS optimization (Dynamics and Control)}
\boldsymbol{x}_{k+1} = \varphi_{\Delta t_k}(\boldsymbol{x}_k; \boldsymbol{u}_k)
\end{equation}
where $\Delta t_k = t_{k+1}-t_k$. The distribution should also satisfy any inequality or equality constraints imposed on any statistical moments at any time index:
\begin{equation}\label{eq: DS optimization (Moment Constraints)}
h_k\left(\mathbb{E}[f^{(m)}(\boldsymbol{X}_k)] \right)\leq 0
,\qquad
g_k\left(\mathbb{E}[f^{(m)}(\boldsymbol{X}_k)] \right)= 0
\end{equation}

The optimization problem is presented in Problem~\ref{prob: DS optimization}:
\begin{problem} \label{prob: DS optimization}
Statistical Moment Steering
\begin{subequations}
\begin{align}
\min_{
\boldsymbol{X}_k,\boldsymbol{U}_k}
\quad
& J \\
\st \quad
&\mathrm{Eq.~\eqref{eq: DS optimization (Initial Distribution)}} \\
&\mathrm{Eq.~\eqref{eq: DS optimization (Dynamics and Control)}} , \;\forall k \in \mathbb{Z}_{0:N-2}\\
&\mathrm{Eq.~\eqref{eq: DS optimization (Moment Constraints)}} , \;\forall k \in \mathbb{Z}_{0:N-1}
\end{align}
\end{subequations}
\end{problem}
Accordingly, Problem~\ref{prob: DS optimization}'s focus on constraining the distribution's moments leads this paper to name it \emph{statistical moment steering (SMS)}. This paper investigates a linear feedback control in the following form at $t_k$:
\begin{equation}\label{eq: control law}
    \boldsymbol{U}_k = \bar{\boldsymbol{u}}_k + K_k(\boldsymbol{X}_k - \boldsymbol{\mu}_k) 
\end{equation}
where $\bar{\boldsymbol{u}}_k$ is the feedforward control and $K_k$ is a feedback gain. This leads to an objective function to accommodate the feedback policy: $\min_{\boldsymbol{X}_k,\bar{\boldsymbol{u}}_k, K_k} \ J$.

\section{Liouville Equation and Statistical Moments}\label{sec: Liouville}
The Liouville equation is a partial differential equation that governs the evolution of a probability density function in deterministic dynamics \cite{Zwanzig-StatMech, Boussouira-PDE-control}. Let $p(\boldsymbol{x},t)$ be the pdf of $\boldsymbol{X}$ at time $t$. The Liouville equation is:
\begin{equation}
    \partial_t p(\boldsymbol{x},t) +  \nabla\cdot (p(\boldsymbol{x},t)f(\boldsymbol{x},\boldsymbol{u})) = 0
\end{equation}
A relationship between the pdf at $t$ and $t_0$ can be made in terms of the Jacobian of the flow map.
\begin{proposition}\label{proposition: pdf t0 to t}
The characteristic solution of the Liouville Equation is:
\begin{equation}
    p(\boldsymbol{x},t) = p(\boldsymbol{x}_0,t_0)
    \det(\Phi(t,t_0;\Pi))^{-1}
\end{equation}
\end{proposition}
\begin{proof}
    Found in \cite{Boussouira-PDE-control}.
\end{proof}

With Proposition~\ref{proposition: pdf t0 to t}, the following lemma can be made regarding the time evolution of a distribution's statistical moments:
\begin{lemma}\label{lemma: stat moment at t}
The moment propagation under the flow map is:
\begin{equation}\label{eq: stat moment at t}
\int_{\mathbb{R}^n} f^{(m)}(\boldsymbol{x})
p(\boldsymbol{x},t)
\,\mathrm{d}\boldsymbol{x}
=
\int_{\mathbb{R}^n} f^{(m)}(\varphi_t(\boldsymbol{x}_0;\Pi))
p(\boldsymbol{x}_0,t_0)
\,\mathrm{d}\boldsymbol{x}_0
\end{equation}
\end{lemma}

\begin{proof}
The statistical moments at $t$ can be expressed with the pdf at $t_0$ using Proposition~\ref{proposition: pdf t0 to t}:
\begin{equation}
\int_{\mathbb{R}^n} f^{(m)}(\boldsymbol{x})
p(\boldsymbol{x},t)
\,\mathrm{d}\boldsymbol{x}
=
\int_{\mathbb{R}^n} f^{(m)}(\boldsymbol{x})
p(\boldsymbol{x}_0,t_0)
\det(\Phi(t,t_0;\Pi))^{-1}
\,\mathrm{d}\boldsymbol{x}
\end{equation}
Performing a change of variables in the integral: 
\begin{equation}
\begin{aligned}
\int \ldots \int 
    \mathrm{d}x_1\ldots \mathrm{d}x_{n_x} &=
    \int \ldots \int
    \left|
    \det \left(
    \frac{\partial \boldsymbol{x}(t)}{\partial \boldsymbol{x}_0}
    \right)
    \right|
    \mathrm{d}x_{0,1}\ldots \mathrm{d}x_{0,n_x}\\
     &=
     \int_{\mathbb{R}^n} 
     \left|
     \det(\Phi(t,t_0;\Pi))
     \right|
     \mathrm{d}\boldsymbol{x}_0
 \end{aligned}
\end{equation}
Since for continuous-time dynamical systems $\det(\Phi(t,t_0;\Pi)) > 0$ \cite{Meiss-DST}, this results in the cancellation of the determinant in the integral. By also substituting $\boldsymbol{x} = \varphi_t(\boldsymbol{x}_0;\Pi)$, it can be seen that the statistical moments at $t$ can be expressed by the statistical moment at $t_0$ with the flow map to $t$ in Eq.~\eqref{eq: stat moment at t}.
\end{proof}

\subsection{Reachability of Moments}
For an unconstrained deterministic linear system that is controllable over the prescribed horizon, any positive-definite terminal covariance is reachable from a positive-definite initial covariance \cite{Liu-CS-controllability}. Analogous nonlinear or non-Gaussian reachability statements are generally more difficult due to the lack of a closed-form representation of the distribution. Accordingly, rather than seeking a complete characterization of nonlinear distributional reachability, this paper provides the following insights into statistical moment reachability:
\begin{lemma} \label{lemma: linear-moment-coupling} 
Consider the discrete-time linear system $\boldsymbol{X}_{k+1} =A_k\boldsymbol{X}_k+B_k\boldsymbol{U}_k $ under the affine control policy in Eq.~\eqref{eq: control law}. Consider control to an arbitrary final time $t_{N-1}$. Then, 
\begin{equation} \label{eq: affine mapping of Xk}
\boldsymbol{X}_{N-1} =\Phi_{N-1}\boldsymbol{X}_0+\boldsymbol{b}_{N-1}, \qquad \boldsymbol{b}_{N-1} =\boldsymbol{\mu}_{N-1}-\Phi_{N-1}\boldsymbol{\mu}_0 
\end{equation} where $\Phi_{N-1} =(A_{N-2}+B_{N-2}K_{N-2})\cdots(A_0+B_0K_0)$. Consequently, the reachable terminal moments are restricted to those induced jointly by the same affine map $(\Phi_{N-1},\boldsymbol{b}_{N-1})$; thus, not every prescribed collection of terminal moments is reachable.
\end{lemma} 

\begin{proof} 
Defining $\boldsymbol{Z}_k=\boldsymbol{X}_k-\boldsymbol{\mu}_k$ gives $\boldsymbol{Z}_{k+1} =(A_k+B_kK_k)\boldsymbol{Z}_k.$ Thus, $\boldsymbol{Z}_{N-1}=\Phi_{N-1}\boldsymbol{Z}_0$, and substituting back $\boldsymbol{X}_k$ yields Eq.~\eqref{eq: affine mapping of Xk}. Hence, for every moment function, 
\begin{equation}
    \mathbb{E}\!\left[f^{(m)}(\boldsymbol{X}_{N-1})\right] = \mathbb{E}\!\left[ f^{(m)}(\Phi_{N-1}\boldsymbol{X}_0+\boldsymbol{b}_{N-1})\right]
\end{equation}
So all moments depend on the same affine map. To show that this restriction excludes specific admissible moments, consider the scalar case and define the skewness $\gamma(X) = {\mathbb{E}\!\left[(X-\mathbb{E}[X])^3\right]} {\operatorname{Var}(X)^{-3/2}}$. For $X_{N-1}=\Phi_{N-1} X_0+b_{N-1}$ with $ \Phi_{N-1} \neq0$, $ \gamma(X_{N-1}) = \Phi_{N-1}^3|\Phi_{N-1}|^{-3}\gamma(X_0) = \mathrm{sign}(\Phi_{N-1})\gamma(X_0)$. Hence, $|\gamma(X_{N-1})|=|\gamma(X_0)|$; an affine map can only preserve or reverse the sign of the initial skewness. If $\Phi_{N-1}=0$, then $\mathrm{Var}(X_{N-1})=0$. Therefore, any target with positive variance and $|\gamma(X_{N-1})|\neq|\gamma(X_0)|$ is unreachable, meaning not every prescribed collection of terminal moments is reachable.
\end{proof}

\begin{corollary}\label{col: nonlinear dist control}
Even with the affine control policy in Eq.~\eqref{eq: control law}, nonlinear dynamics can generate higher-order terminal moments that are unreachable through an affine terminal map. 
\end{corollary}
\begin{proof}
From Lemma~\ref{lemma: stat moment at t}, if the flow to $\varphi_f$ is non-affine on the support of $\boldsymbol{X}_0$, no constant $(\Phi_f,\boldsymbol{b}_f)$ represents this flow. Consequently, the affine moment restrictions in Lemma~\ref{lemma: linear-moment-coupling} need not hold. For example, an affine scalar map preserves the magnitude of skewness, whereas a nonlinear flow can change it. Such a terminal skewness can be made reachable, demonstrating the additional control authority provided by the nonlinear dynamics.
\end{proof}

\section{Optimal Statistical Moment Steering}\label{sec: OSMS}
Given the difficulties of solving the Liouville equation as a partial differential equation under control, this paper introduces the cubature approximation approach to controlling non-Gaussian distributions. Namely, this paper leverages the relationships found in Section~\ref{sec: Liouville} to provide a theoretical basis for SMS. Using the initial cubature rule from Eq~\eqref{eq: cubature rule} on the initial density: 
\begin{equation}\label{eq: initial cubature point}
\mathcal{Q}\left[p(\boldsymbol{x}_0,t_0)\right]
=
\{(\boldsymbol{x}^{(i)}_0,w_i)\}_{i\in \mathbb{Z}_{1:n_s}}
\end{equation}
Applying Lemma~\ref{lemma: stat moment at t} and Eq.~\eqref{eq: sigma point sum} gives
\begin{equation}
    \begin{aligned}
\int_{\mathbb{R}^n} f^{(m)}(\boldsymbol{x})
p(\boldsymbol{x},t)
\,\mathrm{d}\boldsymbol{x}
&=
\int_{\mathbb{R}^n} f^{(m)}(\varphi_t(\boldsymbol{x}_0;\Pi))
p(\boldsymbol{x}_0,t_0)
\,\mathrm{d}\boldsymbol{x}_0 =\\
\mathbb{E}[f^{(m)}(\varphi_t(\boldsymbol{X}_0;\Pi))]
&\approx 
\sum_i w_i f^{(m)}(\varphi_t(\boldsymbol{x}^{(i)}_0;\Pi))
    \end{aligned}
\end{equation}

Thus, the moments at time $t$ can be approximated by propagating the initial cubature points through the flow map:
\begin{equation}\label{eq: flow dt}
\boldsymbol{x}^{(i)}_{k+1} = \varphi_{\Delta t_k}(\boldsymbol{x}^{(i)}_k; \boldsymbol{u}_k),
\qquad
\forall i\in\mathbb{Z}_{1:n_s}
\end{equation}
Subsequently, the moment constraints are then expressed in terms of the cubature points:
\begin{equation}\label{eq: cubature point moments constraints}
h_k\left(
\sum_i w_i f^{(m)}(\boldsymbol{x}^{(i)}_k)
\right)\leq 0
,\qquad
g_k\left(
\sum_i w_i f^{(m)}(\boldsymbol{x}^{(i)}_k)
\right)= 0
\end{equation}

Thus, Problem~\ref{prob: DS optimization} can be written in terms of the approximation with cubature points:
\begin{problem}\label{prob: OSMS Noncovex}
SMS with Cubature Points
\begin{subequations}
\begin{align}
\min_{\boldsymbol{x}^{(i)}_k,\bar{\boldsymbol{u}}_k,K_k}
\quad
& J \\
\st \quad
&\mathrm{Eq.~\eqref{eq: initial cubature point}} \\
&\mathrm{Eq.~\eqref{eq: flow dt}} , \;\forall k \in \mathbb{Z}_{0:N-2}\\
&\mathrm{Eq.~\eqref{eq: cubature point moments constraints}} , \;\forall k \in \mathbb{Z}_{0:N-1}
\end{align}
\end{subequations}
\end{problem}

\begin{remark}\label{remark: linear-moment-controllability}
For linear systems, Lemma~\ref{lemma: linear-moment-coupling} showed that not all moments are reachable, implying the cubature points cannot be moved independently or rearranged arbitrarily; their terminal configuration must be an affine image of their initial configuration. Conversely in nonlinear systems, Corollary~\ref{col: nonlinear dist control} shows that there can be additional control over higher-order moments, even with an affine control policy.
\end{remark}

\subsection{Measures of Risk}
If $\mathcal{Z}$ is a random variable associated with risk, a risk measure is defined as $\rho:\mathcal{L} \rightarrow \mathbb{R}$ where $\mathcal{Z}\in \mathcal{L}$ and $\mathcal{L}$ is the space of random variables that represent risk over some fixed time interval \cite{coherent-risk-measure-original}. In SMS, risk is associated directly with the stochastic state and control $\{\boldsymbol{X}_k,\boldsymbol{U}_k\}_{k\in\mathbb{Z}_{0:N-1}}$. Although risk is stochastic, risk measures map it to a single deterministic quantity for optimization. The objective function to the SMS optimization problem is then 
\begin{equation}
    J = \rho\left(\{\boldsymbol{X}_k,\boldsymbol{U}_k\}_{
k\in\mathbb{Z}_{0:N-1}}\right) 
\end{equation}
A common risk measure for CS is the expected quadratic fuel cost \cite{Liu-CS, Ridderhof-CS}. Given the control policy in Eq.~\eqref{eq: control law}, it can be shown that this objective can be written in terms of the feedforward term and the state covariance:  
\begin{equation}\label{eq: quadratic fuel}
\begin{aligned}
\mathbb{E}
\left[
    \sum_{k=0}^{N-2}
    \boldsymbol{U}_k^\top
    R_k
    \boldsymbol{U}_k
\right]=
\sum_{k=0}^{N-2}
\left[
    \bar{\boldsymbol{u}}_k^\top
    R_k
    \bar{\boldsymbol{u}}_k
    +
    \left\|
        R_k^{1/2}
        K_k
        P_k^{1/2}
    \right\|_{\mathrm{F}}^2
\right]
\end{aligned}
\end{equation}
with control weight matrix $R_k \succ0$. 

Another risk measure is Conditional Value at Risk (CVaR), defined as $\text{CVaR}_{\alpha}(\mathcal{Z}) \triangleq \mathbb{E}[\mathcal{Z}|\mathcal{Z} \geq Q_{\mathcal{Z}}(\alpha)]$ where $Q_{\mathcal{Z}}(\alpha)$ is the quantile function of $\mathcal{Z}$ evaluated at a given $\alpha$ \cite{coherent-risk-measure-original}. This can also be expressed as a convex function in terms of $\mathcal{Z}$, which this paper uses as an objective to minimize the sum of control norms:
\begin{equation}\label{eq: CVar fuel}
\begin{aligned}
    &\text{CVaR}_{\alpha}\left(\sum_{k=0}^{N-2} ||\boldsymbol{U}_k||_2\right) =
    \\
    &\inf_{\tau}\ 
    \left[ \tau + \frac{1}{1-\alpha} \sum_i w_i\left(\sum_{k=0}^{N-2} ||\boldsymbol{u}^{(i)}_k||_2-\tau\right)^+ \right]
\end{aligned}
\end{equation}
where $\tau \in \mathbb{R}$, $(z)^+ = \max(z,0)$, and control cubature points $\boldsymbol{u}^{(i)}_k = \bar{\boldsymbol{u}}_k + K_k(\boldsymbol{x}^{(i)}_k - \boldsymbol{\mu}_k)$. A notable characteristic of using CVaR as an objective function is its ability to minimize tail risk, and this characteristic has been used in previous stochastic control problems \cite{Echigo-Cvar, Qi-risk}. 

\section{SMS via Sequential Convex Programming}
This section outlines how Problem~\ref{prob: OSMS Noncovex} can be solved with sequential convex programming (SCP). The following subsections comment on the convexity of the elements and detail methods to convexify them. A GitHub implementation of the numerical example in Section~\ref{sec: results} is provided in the Appendix.

\subsection{Centralized Cubature Points}
Define a new ``centralized'' random variable $\boldsymbol{Z}_k$ such that 
\begin{equation}
    \boldsymbol{Z}_k = \boldsymbol{X}_k - \mathbb{E}[\boldsymbol{X}_k]
\end{equation}
where it can be seen that $\mathbb{E}[\boldsymbol{Z}_k] = 0$. As a result, $m$-th central moment of $\boldsymbol{X}_k$ is equivalent to the $m$-th raw moment of $\boldsymbol{Z}_k$. This fact is leveraged to simplify expressions. Let the state and centralized cubature points at the $k$-th instance be aggregated into a single vector denoted by $\mathbf{x}_k  \in\mathbb{R}^{n_x n_s }$ and $\mathbf{z}_k  \in\mathbb{R}^{n_x n_s }$ respectively.  
\begin{equation}
    \begin{aligned}
 \mathbf{x}_k =
 \begin{bmatrix}
     \boldsymbol{x}_k^{(1)} \\
     \vdots
     \\
     \boldsymbol{x}_k^{(n_s)} 
 \end{bmatrix}
&\qquad&
\mathbf{z}_k =
 \begin{bmatrix}
     \boldsymbol{z}_k^{(1)} \\
     \vdots
     \\
     \boldsymbol{z}_k^{(n_s)} 
 \end{bmatrix}
    \end{aligned}
\end{equation}
Let $\boldsymbol{x}^{(i)}_k = E_i \mathbf{x}_k$, $\boldsymbol{z}^{(i)}_k =E_i \mathbf{z}_k$, where $E_i$ is a matrix that selects the $i$-th cubature point from the aggregated cubature point vector. It is seen that
\begin{equation}
    \boldsymbol{z}_{k}^{(i)} =\boldsymbol{x}_{k}^{(i)} - \sum_i w_i \boldsymbol{x}_{k}^{(i)}
\end{equation}
With the aggregated cubature point form,
\begin{equation} \label{eq: centralized cubature points}
\begin{aligned} 
    \mathbf{z}_{k} &= \mathbf{x}_{k} - \bar{I}\sum_i w_i E_i \mathbf{x}_{k} 
    = \left(I_{n_x n_s} - \bar{I}\sum_i w_i E_i \right)
    \mathbf{x}_{k} = A^{(z)} \mathbf{x}_{k}
\end{aligned}
\end{equation}
where $\bar{I} = [I_{n_x}\;I_{n_x}\ldots\;I_{n_x}]^\top$ and $A^{(z)}$ is the exact linear relationship between the aggregated vectors.

\subsection{Convexification of Statistical Moments}
Let $(^*)$ denote a reference value, and $\delta$ denote the deviation from the reference value. For example, $\boldsymbol{x}^{(i)*}_k$ denotes the reference for the $i$-th cubature point, and $\mathbf{x}^{*}_k$ denotes the aggregated cubature points at $k$-th instance. Subsequently, $\boldsymbol{x}^{(i)}_k = \boldsymbol{x}^{(i)*}_k + \delta \boldsymbol{x}^{(i)}_k$ and $\mathbf{x}_k = \mathbf{x}^{*}_k + \delta\mathbf{x}_k$. 

\subsubsection{Mean}
The state mean at $t_k$ is defined as $\boldsymbol{\mu}_k = \mathbb{E}[\boldsymbol{X}_k]$. Calculating with the state's cubature points,
\begin{equation}\label{eq: convex, mean}
    \begin{aligned}
\boldsymbol{\mu}_k= \sum_i w_i \boldsymbol{x}^{(i)}_k = \sum_i w_i E_i \mathbf{x}_k= \left(\sum_i w_i E_i \right) \mathbf{x}_k = A^{(\mu)} \mathbf{x}_k 
    \end{aligned}
\end{equation}
where $A^{(\mu)}$ is the exact linear relationship between aggregated cubature points and the mean. 

\subsubsection{Covariance}
The covariance at $t_k$ is written in terms of the cubature points:
\begin{equation}\label{eq: Pk_full} 
    P_k = \sum_i w_i (\boldsymbol{z}^{(i)}_k)(\boldsymbol{z}^{(i)}_k)^\top
\end{equation}
As per Assumption~\ref{assum: cubature points assumption}, the weights are all strictly positive. Then, it can be seen that $P_k^{1/2}$ can be found by
\begin{equation}\label{eq: Pk_sqrt} 
P_k^{1/2} 
    =
    \begin{bmatrix}
    \sqrt{w_1}  \boldsymbol{z}^{(1)}_k
    &
    \sqrt{w_2}  \boldsymbol{z}^{(2)}_k
    &
    \ldots
    \end{bmatrix}
\end{equation}
where $P_k^{1/2} \in \mathbb{R}^{n_x \times n_s}$. Using square-root covariance has two main advantages. Firstly, square-root covariance is affine with respect to the cubature points, while the full covariance in Eq.~\eqref{eq: Pk_full} is not. The square-root covariance can then be a convex constraint on the largest eigenvalue of the full covariance:
\begin{equation}\label{eq: Pk_sqrt norm} 
    \sqrt{\lambda_{\max}(P_k)} = \left\|{P_k^{1/2}}\right\|_2
\end{equation}
The second main advantage is that the value of the square-root covariance is on the same numerical order as the cubature points. This poses better numerical properties than the full covariance as it requires the cubature points to be squared. 

\subsubsection{Higher-ordered Moments}\label{sec: High-ordered Moments}
Under Assumption~\ref{assum: cubature points assumption}, the cubature rule has sufficient degree to impose the desired higher-order moment constraints. Consider the $m$-th standardized moment along the $j$-th axis at $t_k$:
\begin{equation}
    {^mC_{j,k}} = \mathbb{E}\left[\left(\frac{X_{j,k} - \mu_{j,k}}{\sigma_{j,k}}\right)^m\right] =
    \mathbb{E}\left[Z_{j,k}^m\right]
    \mathbb{E}\left[Z_{j,k}^2\right]^{-m/2}
\end{equation}
where $^mC_{k} \in \mathbb{R}^{n_x}$. In terms of cubature points,
\begin{equation}
\mathbb{E}\left[Z_{j,k}^m\right] =  \sum_i w_i (e_j E_i\mathbf{z}_k)^m
\qquad
\mathbb{E}\left[Z_{j,k}^2\right]
=
\sum_i w_i (e_j E_i\mathbf{z}_k)^2
\end{equation}
where $z^{(i)}_{j,k} =e_j \boldsymbol{z}^{(i)}_k$  and $e_j$ is a matrix that selects the $j$-th component of the vector. The linearized equation for the $m$-th standardized moment is then:
\begin{equation}\label{eq: convex, HO moments}
    {^m\boldsymbol{C}_{k}} \approx {^m\boldsymbol{C}_{k}}^* + A^{(^m C)} \biggr\rvert_{\boldsymbol{z}_k^*}\delta \mathbf{z}_k     \qquad
\end{equation}
where $A^{(^m C)}$ is the linearized relationship between the centralized aggregated vector and any $m$-th standardized moment.

\subsubsection{Linearization of Flow Map}
From Eq.~\eqref{eq: flow dt}, the pdf's dynamics under control are directly approximated with the flow of the cubature points themselves. Subsequently, the dynamics of the cubature points can be linearized in this fashion:
\begin{equation}\label{eq: cubature point dynamics}
    \begin{aligned}
\boldsymbol{x}^{(i)}_{k+1}\approx A^{(i)}_k\boldsymbol{x}^{(i)}_{k} + B^{(i)}_k\boldsymbol{u}^{(i)}_k + \boldsymbol{c}^{(i)}_k,
\qquad
\forall i\in \mathbb{Z}_{1:ns}
    \end{aligned}
\end{equation}
where $A^{(i)}_k$, $B^{(i)}_k$, and $\boldsymbol{c}^{(i)}_k$ are linearized system matrices for each distinct reference cubature point $\boldsymbol{x}^{(i)*}_k$. The control mapping matrix depends on how the control is modeled. In this paper, impulsive control is considered:
\begin{equation}\label{eq: impulsive flow}
    \varphi_{\Delta t_k}(\boldsymbol{x}^{(i)}_k;\boldsymbol{u}^{(i)}_k) = \varphi^{0}_{\Delta t_k}(\boldsymbol{x}^{(i)}_k+B\boldsymbol{u}^{(i)}_k)
\end{equation}
where $\varphi^{0}_{\Delta t_k}$ denotes uncontrolled flow. Then with chain rule
\begin{equation}
    B^{(i)}_k = A^{(i)}_kB \qquad (\text{Impulsive Control})
\end{equation}
As a note, this method is not limited to impulsive control and can be extended to other control schemes, such as zero-order hold approximations of continuous control input, provided that the linearization is carried out at each distinct reference cubature point. It is important to emphasize that each linearized system matrix is with respect to its own cubature point. For nonlinear systems, it is generally true that:
\begin{equation}
    A^{(i)}_k \neq A^{(j)}_k
    \quad
  B^{(i)}_k \neq B^{(j)}_k
  \quad
    \boldsymbol{c}^{(i)}_k \neq \boldsymbol{c}^{(j)}_k
    \quad 
    \text{when $i\neq j$}
\end{equation}
This enables each cubature point’s dynamics to be locally approximated as linear within the convex optimization, while still capturing the non-Gaussian characteristics of the overall distribution it represents as per Corollary~\ref{col: nonlinear dist control}. The control law in Eq.~\eqref{eq: control law} is nonconvex, and linearized accordingly:
\begin{equation}\label{eq: control convex}
\begin{aligned}
    \boldsymbol{u}^{(i)}_k 
    &\approx 
    (\bar{\boldsymbol{u}}^*_k +\delta\bar{\boldsymbol{u}}_k )
    +
    K^*_k \boldsymbol{z}^{(i)*}_k 
    +
    \delta K_k \boldsymbol{z}^{(i)*}_k 
    +
    K^*_k \delta \boldsymbol{z}^{(i)}_k 
\end{aligned}
\end{equation}

\subsubsection{Convexified Risk Measures}
For convex optimization, the objective function must also be convex. Let $\widehat{J}$ be a general convexified objective function. The quadratic risk measure from Eq.~\eqref{eq: quadratic fuel} is convexified as:
\begin{equation}\label{eq:quadratic fuel convexified}
\begin{aligned}
\mathbb{E}
\left[
    \sum_{k=0}^{N-2}
    \boldsymbol{U}_k^\top
    R_k
    \boldsymbol{U}_k
\right]\approx
\sum_{k=0}^{N-2}
\Bigg[
&
\left(
\bar{\boldsymbol{u}}^{*}_k
+\delta\bar{\boldsymbol{u}}_k
\right)^{\top}
\boldsymbol{R}_k
\left(
\bar{\boldsymbol{u}}^{*}_k
+\delta\bar{\boldsymbol{u}}_k
\right)
\\
&+
\left\|
\boldsymbol{R}_k^{1/2}
\left(
\boldsymbol{K}^{*}_k
+\delta\boldsymbol{K}_k
\right)
\left(\boldsymbol{P}^{*}_k\right)^{1/2}
\right\|_{\mathrm{F}}^{2}
\Bigg]
\end{aligned}
\end{equation}
It should be noted that the first-order variation $\delta P_k^{1/2}$ is omitted because its inclusion degraded numerical robustness in empirical tests.

For CVaR, although Eq.~\eqref{eq: CVar fuel} is convex in the control input, the control is a nonconvex function of the optimization variables. Therefore, its first-order approximation from Eq.~\eqref{eq: control convex} is substituted into the equivalent convex CVaR formulation.

\subsection{\texttt{SCvx*}}
The chosen SCP algorithm is \texttt{SCvx*} \cite{Oguri-SCVX} for its favorable convergence guarantees. With the previous subsections, Problem~\ref{prob: OSMS Noncovex} can be written as a convex subproblem:
\begin{problem}  \label{prob: OSMS convex}
Convexified Problem~\ref{prob: OSMS Noncovex}
\begin{subequations}
\begin{align}
\min_{\delta\mathbf{x}_k,\delta\bar{\boldsymbol{u}}_k,\delta K_k}
\quad
& \widehat{J} \\
\st \quad
&\mathrm{Eq.~\eqref{eq: initial cubature point}} \\
&\mathrm{Eq.~\eqref{eq: cubature point dynamics}} , \;\forall k \in \mathbb{Z}_{0:N-2}\\
&\mathrm{Eqs.~\eqref{eq: convex, mean},\eqref{eq: Pk_sqrt},\eqref{eq: convex, HO moments}}, \;\forall k \in \mathbb{Z}_{0:N-1}
\end{align}
\end{subequations}
\end{problem}
For concision, the integration of Problem~\ref{prob: OSMS convex} in the \texttt{SCvx*} framework is omitted, but demonstrations of it can be found in \cite{Qi-Stat-Moment-Steering} as well as GitHub code for Section~\ref{sec: results} in the Appendix. Iteratively solving Problem~\ref{prob: OSMS convex} with the augmented Lagrangian modifications from \texttt{SCvx*} results in the solution for Problem~\ref{prob: OSMS Noncovex}. Let $\mathcal{S}$ be the set of \texttt{SCvx*} parameters. Then, a pseudocode of SMS in \texttt{SCvx*} is found in Algorithm~\ref{algo: SMS}.
\begin{algorithm}
\caption{Statistical Moment Steering via $\mathrm{SCvx}^{*}$}
\label{algo: SMS}
\begin{algorithmic}[1]

\State \textbf{Input:}
$\mathbf{x}_0$, $\{\mathbf{x}_k,\bar{\boldsymbol{u}}^{*}_k,K^{*}_k\}$, $\mathcal{S}$.
\State \textbf{Output:}
$\bar{\boldsymbol{u}}_k$,
$\boldsymbol{K}_k$, and
$\mathbf{x}_k$.
\State Set $\ell=0$ and initialize $\mathrm{SCvx}^{*}$ parameters.

\While{
convergence criteria not met
}
    \State Set $\ell\gets\ell+1$.
    \State Obtain linearization matrices.
    \State Solve Problem~\ref{prob: OSMS convex} with $\mathrm{SCvx}^{*}$ modifications for candidate solution $\{\mathbf{x}_k,\bar{\boldsymbol{u}}_k,{K}_k\}^{+}$.
    \State Compute step acceptance criterion.
    \If{step acceptance criterion met}
        \State Accept solution $\{\mathbf{x}_k,\bar{\boldsymbol{u}}^{*}_k,{K}^{*}_k\}
        \gets
        \{\mathbf{x}_k,\bar{\boldsymbol{u}}_k,{K}_k\}^{+}.
        $
        \If{Lagrange multiplier update condition met}
            \State Update the Lagrange multipliers.
            \State Update the penalty weight and stationarity.
        \EndIf
    \EndIf
    \State Update the trust-region radius.
\EndWhile

\State \textbf{Return:}
$\bar{\boldsymbol{u}}_k$,
${K}_k$, and
$\mathbf{x}_k$.

\end{algorithmic}
\end{algorithm}

\section{Numerical Results}\label{sec: results}
The numerical example provided is a coupled, nonlinear, second-order ordinary differential equation describing a two-dimensional oscillator:
\begin{equation}\label{eq: eom}
    \begin{aligned}
\ddot{x} + 0.25x + 0.75 x^3 + 0.5 (x-y) &= 0 \\
\ddot{y} + 0.25y + 0.75 y^3 + 0.5 (y-x) &= 0
    \end{aligned}
\end{equation}
The state is $[x,y,\dot{x},\dot{y}]^\top$ with an impulsive control mapping $B = [0_{2\times 2}\;\; I_2]^\top$. Eq.~\eqref{eq: eom} generates $\varphi^{0}_{\Delta t_k}$ from Eq.~\eqref{eq: impulsive flow}. Table~\ref{tab: sim parameter} lists the simulation parameters and constraints used in this numerical example. The optimization nodes are evenly spaced in time, yielding a time step of $\Delta t_k = 1.1076$ between nodes. For all examples, the terminal mean is constrained to $\boldsymbol{\mu}_{N-1} =\boldsymbol{\mu}_f$ and the terminal positional covariance is in an inequality constraint $\left\|{H_r P_{N-1}^{1/2}}\right\|_2\leq\sqrt{\lambda_{\max}(P_f)}$ where $H_r = [I_2 \;\; 0_{2\times 2}]$. In the additional skewness-constrained example, the terminal skewness is constrained to be between $-\gamma_f\leq H_r \boldsymbol{\gamma}_{N-1} \leq \gamma_f$ along the positional axes.
\begin{table}[H]
\caption{Simulation Parameters and Constraints}
\label{tab: sim parameter}

\renewcommand{\arraystretch}{1.5} 
\setlength{\tabcolsep}{2pt}       

\centering
\begin{tabular}{|c|c|c|}
\hline
Parameter & Variable &Value \\
\hline
\hline
Optimization Nodes &$N$ & $15$
\\
\hline
Time Horizon &$[t_0,t_f]$ & $[0,15.5059]$
\\
\hline
Initial Mean & $\boldsymbol{\mu}_0$ & $[2,0,0,1.2580]^\top$
\\
\hline
Initial Covariance & $P_0$ & $\mathrm{diag}([2.5,2.5,2.5,2.5]\cdot 10^{-3})$
\\
\hline
Final Mean Constraint & $\boldsymbol{\mu}_f$ & $[-1.9976,0,0,-1.2722]^\top$
\\
\hline
Final $3\sigma$ Constraint & $3\sqrt{\lambda_{\max}(P_f)}$ & $0.25$
\\
\hline
Final Skewness Constraint & $\gamma_f$ & $0.1$
\\
\hline
\end{tabular}
\end{table}
The initial distribution is Gaussian $\boldsymbol{X}_0 \sim \mathcal{N}(\boldsymbol{\mu}_0, P_0)$, and the $4$-th order conjugate unscented transform \cite{Adurthi-CUT} is used as the choice for the cubature rule. This results in 24 points for the 4-dimensional state. Notably, this cubature rule satisfies Assumption~\ref{assum: cubature points assumption}.

\begin{remark}\label{remark: mean in deterministic}
Under uncontrolled dynamics, the value of the initial mean propagates to the value of the terminal mean. See Figure~\ref{fig: ref traj}.
\end{remark}

\begin{figure}[!tbh]
\centering\includegraphics[width=0.45\textwidth]{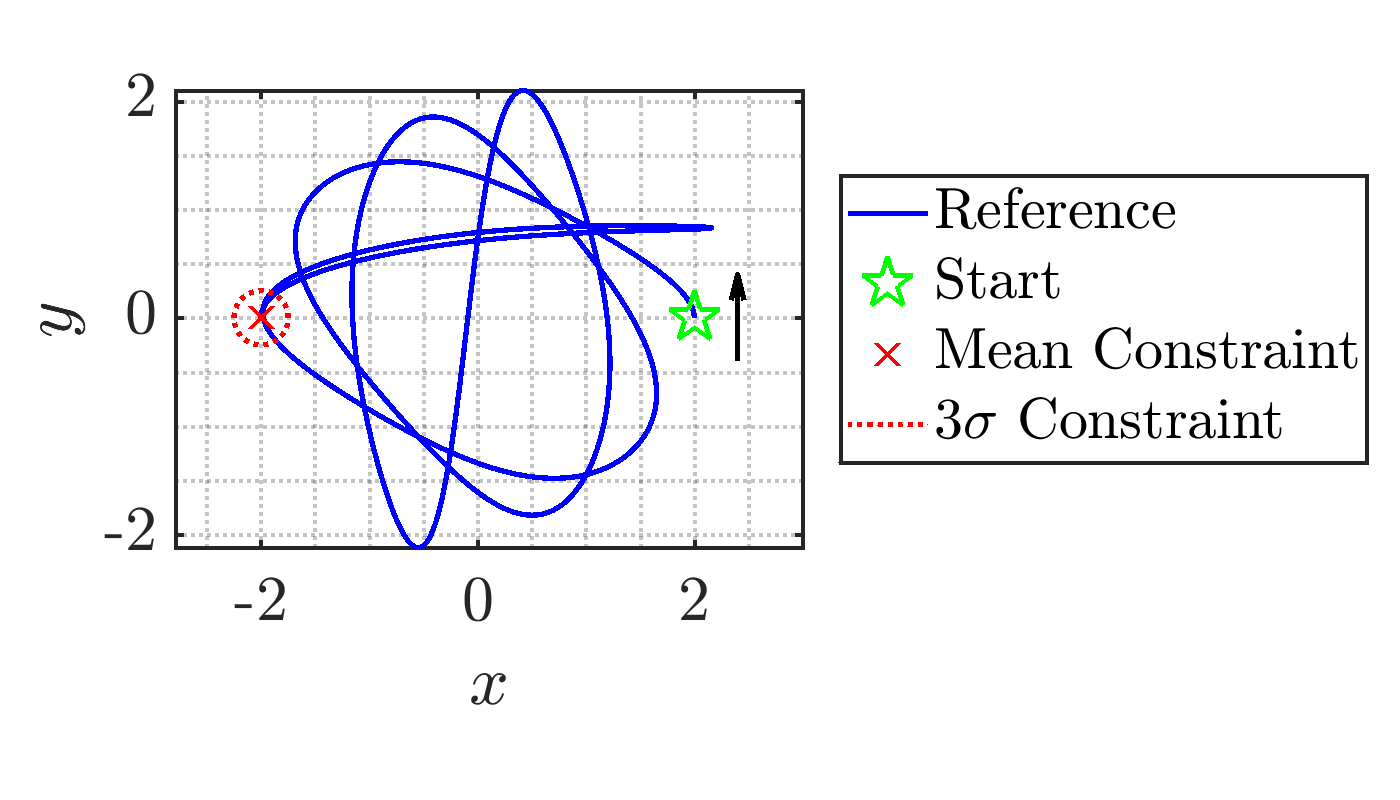}
	\caption{Reference Trajectory.}
	\label{fig: ref traj}
\end{figure}

\subsection{Linear Covariance Steering as Baseline}\label{sec: example, CS}
CS is applied to this example as a baseline comparison. The CS algorithm used is from \cite{Liu-CS} and utilizes a quadratic cost equivalent to Eq.~\eqref{eq: quadratic fuel}. For the CS comparison, the dynamics are linearized about the uncontrolled mean trajectory. Because this trajectory already satisfies the terminal mean constraint (Remark~\ref{remark: mean in deterministic}), no feedforward term is required for linearized CS. The covariance constraints are imposed in the CS and SMS examples. However, from Figure~\ref{fig: final dist CS}, it can be seen that CS is unable to effectively control this distribution for this nonlinear system. Both the final mean and covariance constraints are violated, as the linearized approximations of these moments are inaccurate for this example. Furthermore, it can be seen from the skewness value that this distribution is highly non-Gaussian. 

\begin{figure}[!tbh]
\centering\includegraphics[width=0.4\textwidth]{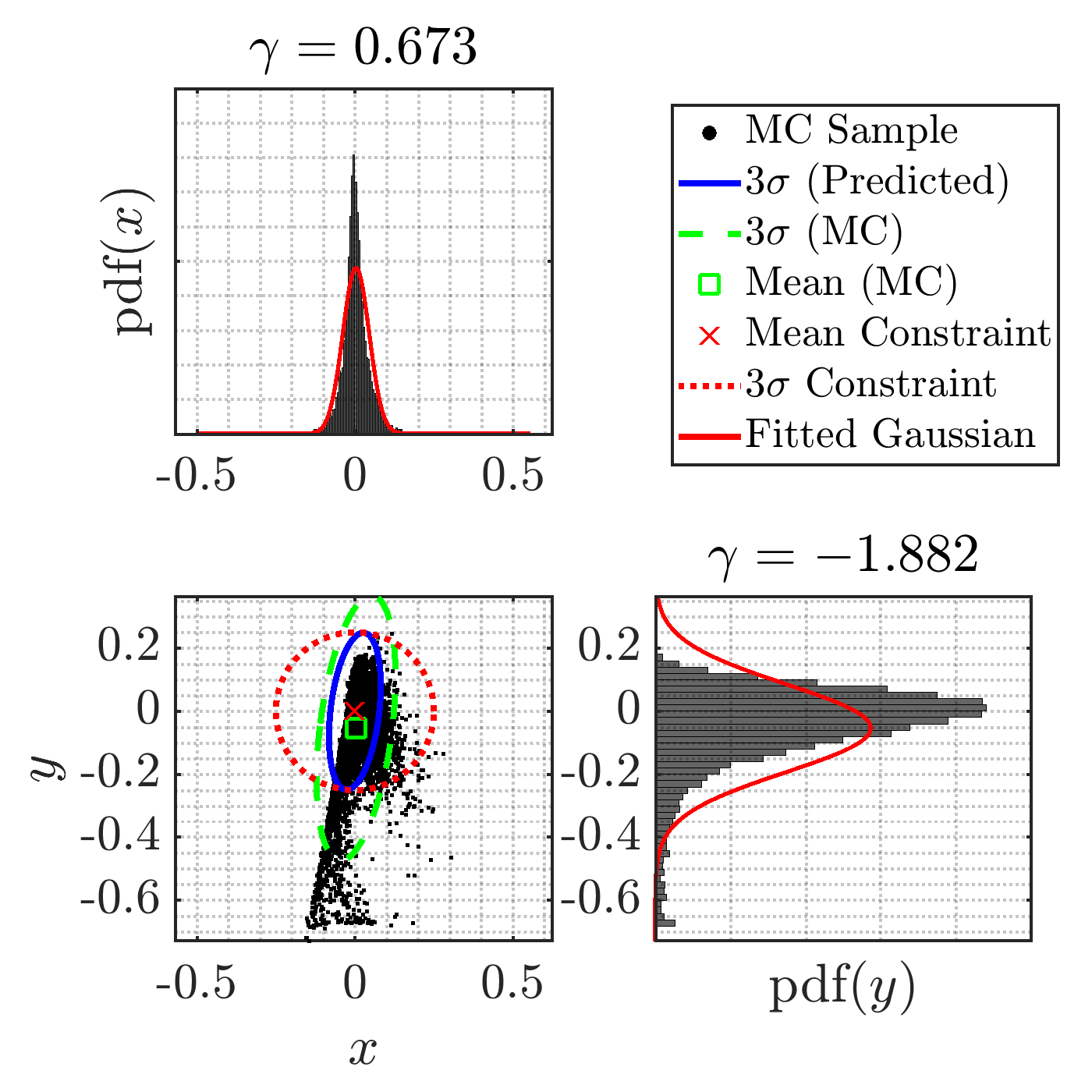}
	\caption{Monte Carlo ($n_{\text{MC}} = 10{,}000$): Joint and marginal distribution of position at final time using linear covariance steering from Liu et al. \cite{Liu-CS}. \emph{Axes origin at predicted mean.}}
	\label{fig: final dist CS}
\end{figure}

\subsection{SMS with Quadratic Cost}\label{sec: example, SMS quad}
\begin{figure}[!t]
\centering\includegraphics[width=0.4\textwidth]{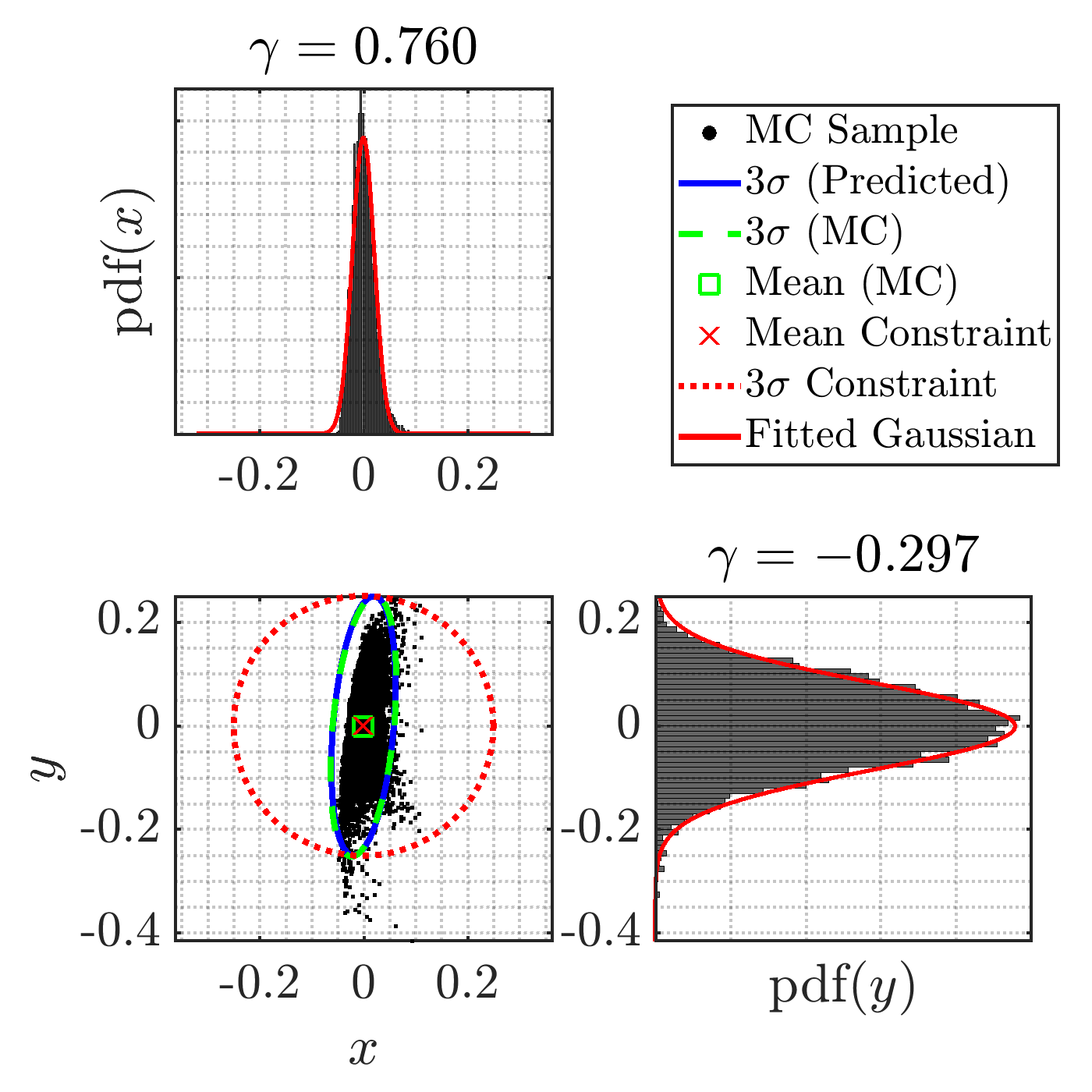}
	\caption{Monte Carlo ($n_{\text{MC}} = 10{,}000$): Joint and marginal distribution of position at final time. \emph{Axes origin at predicted mean.}}
	\label{fig: final dist}
\end{figure}

SMS is first solved with a quadractic objective function (Eq.~\eqref{eq: quadratic fuel}) as a direct comparison with the CS baseline in Section~\ref{sec: example, CS}. The initial reference is calculated using the same method as in \cite{Qi-Stat-Moment-Steering}. The final distribution under statistical moment steering is shown in Figure~\ref{fig: final dist}. Unlike the baseline CS case, the control policy obtained through SMS ensures satisfaction of both the final mean and $3\sigma$ constraint and is observed to have nonzero feedforward control. This is in contrast to Remark~\ref{remark: mean in deterministic} because the evolution of a distribution’s mean cannot, in general, be obtained by simply propagating it through a nonlinear function. As a result, feedforward control is required to properly steer the mean. 

\subsection{Control of Higher-order Moments}\label{sec: example, SMS skew}
\begin{table}[H]
\caption{Terminal Skewness Along Marginal Axes from Monte Carlo ($n_{\text{MC}} = 10{,}000$)}
\label{tab: terminal-skewness}
\renewcommand{\arraystretch}{1.5}
\setlength{\tabcolsep}{4pt}
\centering
\begin{tabular}{|c|c|c|}
\hline
SMS Case & $x$ & $y$ \\
\hline
\hline
Quadratic & 0.760 & -0.297 \\
\hline
Quadratic with Skewness Constraint & -0.127 & -0.033 \\
\hline
\end{tabular}
\end{table}

This section's numerical example is similar to Section~\ref{sec: example, SMS quad} but with an additional skewness constraint to demonstrate Corollary~\ref{col: nonlinear dist control}. Similar to Figure~\ref{fig: final dist}, the resulting terminal distribution for this case satisfies the prescribed mean and covariance constraints, but its plot is omitted for brevity. Importantly, Table~\ref{tab: terminal-skewness} summarizes the effects of the skewness constraint. As shown in the table, imposing the skewness constraint substantially reduced terminal skewness in the nonlinear Monte Carlo simulations. However, the simulated skewness did not fully satisfy the constraint, despite its satisfaction under the cubature prediction. This discrepancy reflects the accuracy of the uncertainty quantification, on which numerical stochastic control and filtering methods generally depend, and could be reduced by using a higher-order cubature rule.

\subsection{Comparison With Different Risk Measures}
\begin{figure}[!tb]
\centering\includegraphics[width=0.48\textwidth]{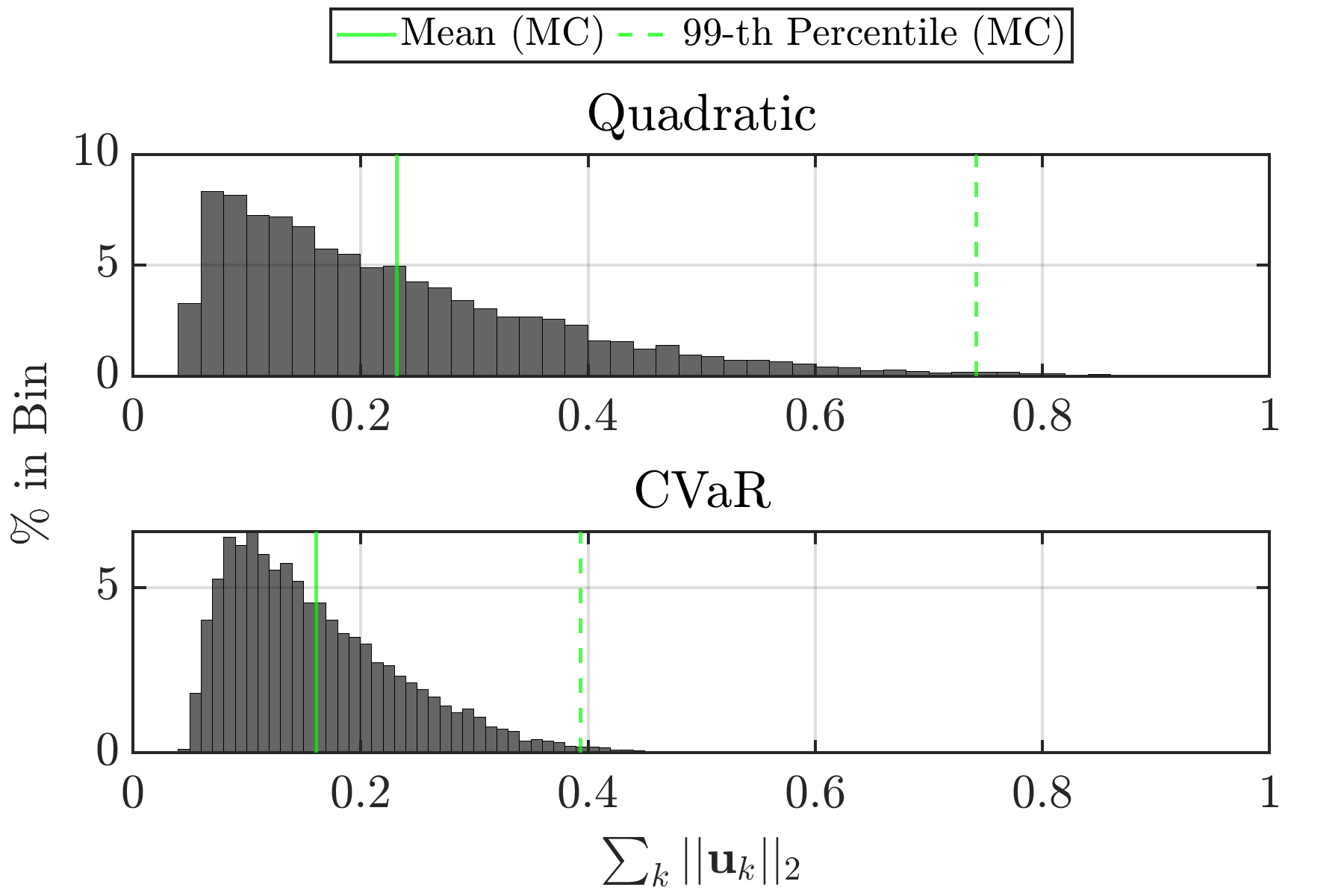}
	\caption{Monte Carlo ($n_{\text{MC}} = 10{,}000$): Distribution of control-norm cost.}
	\label{fig: control dist}
\end{figure}

This section compares Section~\ref{sec: example, SMS quad}'s results with a solution found with a CVaR $\alpha = 0.99$ objective function (Eq.~\eqref{eq: CVar fuel}). The resulting final distribution for this case is visually similar to Figure~\ref{fig: final dist} and is omitted for brevity. The histogram of the control cost is shown in Figure~\ref{fig: control dist}. The figure demonstrates CVaR's tail-flattening properties, as the distribution of quadratic control cost has a much heavier tail. The tighter upper tail under the CVaR objective indicates the reduction of worst-case control costs relative to a conventional quadratic objective. This result also demonstrates the versatility of cubature-based SMS in accommodating risk measures for non-Gaussian distribution steering.

\subsection{Computational Efficiency}
\begin{figure}[!t]
\centering\includegraphics[width=0.48\textwidth]{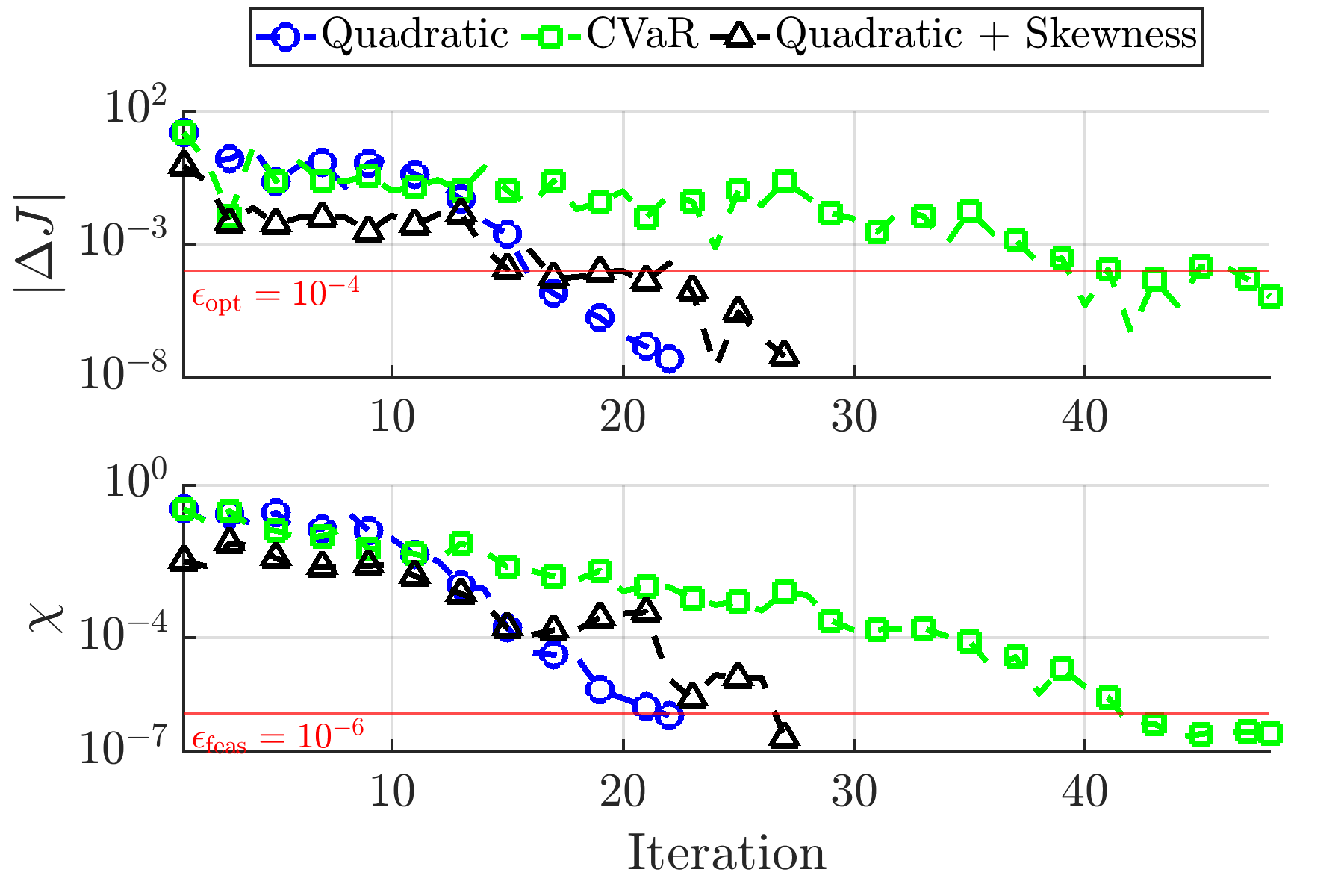}
	\caption{Convergence profile of \texttt{SCvx*} for SMS cases.}
	\label{fig: scp convergence}
\end{figure}

Figure~\ref{fig: scp convergence} shows the convergence profile for all cases. As shown, each case satisfies the tolerances for optimality and feasibility $(\epsilon_{\text{opt}},\epsilon_{\text{feas}})$. The convex subproblems are solved via \texttt{CVX} with \texttt{MOSEK} in \texttt{MATLAB} R2024b, and each iteration of \texttt{SCvx*} takes around 10-15 seconds.


\section{Conclusions}\label{sec: conclusion}
This paper introduces a cubature-based framework for steering non-Gaussian distributions in discrete-time deterministic nonlinear systems subject to statistical moment constraints. Termed statistical moment steering, this paper connects its cubature-point propagation to the Liouville equation, provides insights into moment reachability, and presents a tractable formulation solved through sequential convex programming. Numerical results show that cubature-based statistical moment steering improves non-Gaussian steering accuracy over linearized covariance steering in nonlinear systems, controls higher-order moments, and accommodates various risk measures.

\section{Appendix}
The GitHub link is \url{https://github.com/qi85/statistical_moment_steering}.



\bibliographystyle{IEEEtran}
\bibliography{references}

\end{document}